\documentclass{article}

\usepackage{indentfirst}
\usepackage{amsmath,amsfonts,amsthm,amssymb}
\usepackage{mathrsfs}
\usepackage{amscd}
\usepackage{hyperref}

\def\co{\colon\thinspace}
\DeclareMathAlphabet{\mathsfsl}{OT1}{cmss}{m}{sl}

\newcommand{\Q}{\mathbb{Q}}

\newcommand{\Z}{\mathbb{Z}}

   \RequirePackage{rotating}                   
   \def\Hto{%
       \setbox0=\hbox{$\widehat{\mathit{HM}}$}
       \setbox1=\hbox{$\mathit{HM}$}
       \dimen0=1.1\ht0
       \advance\dimen0 by 1.17\ht1
       \smash{\mskip2mu\raise\dimen0\rlap{%
          \begin{turn}{180}
              {$\widehat{\phantom{\mathit{HM}}}$}
           \end{turn}} \mskip-2mu
                \mathit{HM}
    }{\vphantom{\widehat{\mathit{HM}}}}{}}

\theoremstyle{definition}

\newtheorem{theorem}{Theorem}[section]
\newtheorem{prop}[theorem]{Proposition}
\newtheorem{lemma}[theorem]{Lemma}
\newtheorem{remark}[theorem]{Remark}
\newtheorem{question}[theorem]{Question}
\newtheorem{definition}[theorem]{Definition}

\newtheorem{example}[theorem]{Example}
\newtheorem{conj}[theorem]{Conjecture}
\newtheorem{condition}[theorem]{Condition}

\begin{document}


\title{Virtual Betti numbers and virtual symplecticity of  4-dimensional  mapping tori}

\author{
{\Large Tian-Jun Li}\\{\normalsize School of Mathematics, University of Minnesota}\\
{\normalsize Minneapolis, MN 55455}\\{\small\it Email\/:\quad\rm tjli@math.umn.edu}\\
\\
{\Large Yi Ni}\\{\normalsize Department of Mathematics, Caltech, MC 253-37}\\
{\normalsize 1200 E California Blvd, Pasadena, CA
91125}\\{\small\it Email\/:\quad\rm yini@caltech.edu}
}

\date{}
\maketitle

\begin{abstract}
In this note, we compute the virtual first Betti numbers of $4$--manifolds fibering over $S^1$ with prime fiber. As an application, we show that if such a manifold is symplectic with nonpositive Kodaira dimension, then the fiber itself is a sphere or torus bundle over $S^1$. In a different direction, we prove that if the $3$--dimensional fiber of such a $4$--manifold is virtually fibered then the $4$--manifold is virtually symplectic unless its virtual first Betti number is 1.
\end{abstract}

\section{Introduction}

Given a manifold $M$, the {\it virtual first Betti number} of $M$ is defined to be
$$vb_1(M)=\max\left\{b_1(\widetilde M)\left|\widetilde M \text{ is a finite cover of } M\right.\right\}\in\mathbb Z_{\ge0}\cup\{\infty\}.$$
Virtual first Betti numbers naturally arise in many geometric and topological problems. In many cases $vb_1$ is $\infty$. A classical result of Kojima \cite{Ko} and Luecke \cite{Luecke} says that $3$--manifolds with nontrivial JSJ decompositions have infinite $vb_1$. The recent progress on the Virtually Haken Conjecture \cite{Agol} yields a complete computation of $vb_1$ for $3$--manifolds. For simplicity, we only state the result for closed irreducible $3$--manifolds.

\begin{theorem}[Agol et al.]\label{thm:vb3d}
Suppose that $Y$ is a closed irreducible $3$-manifold, then there are three cases:
\newline (1) If $Y$ is a spherical manifold, then $vb_1(Y)=0$;
\newline (2) If $Y$ is finitely covered by a $T^2$--bundle over $S^1$, then $vb_1(Y)$ is equal to either $1$, or $2$, or $3$, depending on whether the monodromy of the $T^2$--bundle is Anosov, or reducible, or periodic;
\newline (3) In all other situations, $vb_1(Y)=\infty$.
\end{theorem}

The virtual Betti numbers for 4-manifolds which fiber over $2$--manifolds were subsequently computed in \cite{Bay} and \cite{FV5}. In \cite{Bay}, the virtual Betti numbers for  most $4-$manifolds which fiber over  $3$--manifolds were also shown to be $\infty$.
In this paper, we will study the problem for 4-manifolds which fiber over $S^1$, that is, $4$--manifolds which are mapping tori. 

All manifolds we consider are oriented unless otherwise stated.
If $E$ is an $F$--bundle over $B$, then we denote $E=F\rtimes B$. If $B=S^1$, $\varphi$ is the monodromy, then $E=F\rtimes_{\varphi}S^1$.

Our first theorem is a complete computation of $vb_1$ for $X=Y\rtimes S^1$ with $Y$ prime.

\begin{theorem}\label{thm:vb}
Suppose that $X$ is a closed $4$-manifold which fibers over the circle with fiber $Y$. Assume that $Y$ is prime, then there are three cases:
\newline (1) If $Y$ is a spherical manifold, then $vb_1(X)=1$;
\newline (2) If $Y$ is $S^1\times S^2$ or finitely covered by a $T^2$--bundle over $S^1$, then $vb_1(X)\le4$;
\newline (3) In all other cases, $vb_1(X)=\infty$.
\end{theorem}

The $vb_1$ in the above Case (2) is not hard to compute, so we leave it to the reader. 
\begin{remark}
Since Euler characteristic and signature of mapping
tori are both zero (and since both are multiplicative under
coverings), we can conclude that, as in \cite{Bay}, whenever  $vb_1=\infty$  the virtual $b^+$  and $b^-$ are
infinite as well.

Although our Theorem~\ref{thm:vb} considers a different class of fibered $4$--manifolds from  \cite{Bay} and \cite{FV5}, there is  a significant overlap. When the $4$--manifold $X$ is a surface bundle over $T^2$, it also admits a fibration over $S^1$. On the other hand, if $X$ admits a fibration over $S^1$, in many cases (see \cite{BF}) $X$ is finitely covered by a surface bundle over $T^2$. The case of our theorem that is not covered by \cite{Bay} and \cite{FV5} is that $Y$ has at least one Seifert fibered JSJ piece.
\end{remark}



Suppose $X^4=Y^3\times S^1$. It is a classical theorem of Thurston \cite{ThSympl} that if $Y$ fibers over the circle then $X$ has a symplectic structure. Friedl and Vidussi \cite{FV2} proved the converse of Thurston's theorem, namely, if $X$ has a symplectic structure, then $Y$ fibers over the circle. 

Friedl and Vidussi \cite{FV3,FVcons,FV4} also studied the question when a symplectic manifold $X^4$ is a circle bundle over $Y$.

In this paper we study  the following question.

\begin{question}\label{ques:symplMT}
Which symplectic $4$--manifold $X$ fibers over the circle with fiber a connected $3$--manifold $Y$? 
\end{question}

An immediate consequence of Friedl and Vidussi's theorem \cite{FV2} is that $Y$ is fibered  if the monodromy of $X$ is of finite order. 
One may guess that $Y$ fibers over $S^1$ for any monodromy of $X$. However, the next example shows that this is not the case.

\begin{example}
let $N$ be a $3$--manifold which fibers over $S^1$ in two different ways, $p_i\co N\to S^1$, $i=1,2$. Here ``different'' simply means that $[F_1],[F_2]$ are linearly independent in $H_2(N)$, where $F_i$ is the fiber of $p_i$. We also assume that $g(F_i)>1$. There exists a cohomology class $e\in H^2(N)$ such that $e([F_1])=0$ but $e([F_2])\ne0$. Let $q\co X\to N$ be the circle bundle over $N$ with Euler class $e$. Then $p_i\circ q$, $i=1,2$, are two different fibrations of $X$ over $S^1$. Let $Y_i$ be the fiber of $p_i\circ q$, then $Y_i$ is a circle bundle over $F_i$ with Euler class $e([F_i])$. Since $e([F_1])=0$, from the fibration $p_1\circ q$ we can construct a symplectic structure on $X$ \cite{Bou,FGM,BL,FVcons}. As $e([F_2])\ne0$, $Y_2$, the fiber of the second fibration, is not a surface bundle over $S^1$.
\end{example}

The connection between $vb_1$ and Question~\ref{ques:symplMT} is via symplectic Kodaira dimension $\kappa(X)$ (see Section~\ref{Sect:Kod}).
It is easy to see that the Kodaira dimension  of $X$ is at most 1. A theorem of Li \cite{L2} and Bauer \cite{B} asserts that $vb_1(X)\le4$ if the Kodaira dimension of $X$ is zero.

Using Theorem~\ref{thm:vb}, we can answer Question~\ref{ques:symplMT}
when $\kappa(X)\leq 0$ and $Y$ is irreducible.
We have the following classification.

\begin{theorem}\label{thm:ClassifyFiber}
Suppose that $X=Y\rtimes S^1$ is a symplectic $4$--manifold and $Y$ is prime.
If $\kappa(X)=-\infty$, then $Y=S^2\times S^1$ and $X=S^2\times T^2$.
If $\kappa(X)=0$, then $Y$ is a $T^2$--bundle over $S^1$ and $X$ is a $T^2$--bundle over $T^2$.
\end{theorem}

In a different direction, Baykur and Friedl \cite{BF} studied the question when a $4$--dimensional mapping torus is virtually symplectic, namely, finitely covered by a symplectic manifold. Using deep results about virtual fibration of $3$--manifolds, they proved that if $Y$ is irreducible and the JSJ decomposition of $Y$ has only hyperbolic pieces, then $X=Y\rtimes S^1$ is virtually symplectic. We will prove a more general virtual symplecticity theorem.

\begin{theorem}\label{thm:VirtSympl}
Suppose that a closed $3$--manifold $Y$ is finitely covered by $F\rtimes S^1$, $X=Y\rtimes S^1$.

(1) If $g(F)=0$, then $X$ is virtually symplectic and $v\kappa(X)=-\infty$, where $v\kappa$ is the virtual Kodaira dimension defined in Section~\ref{Sect:Kod}.

(2) If $g(F)=1$, then $X$ is virtually symplectic if and only if $vb_1(X)\ge2$. Moreover, if $vb_1\geq 2$ then $v\kappa(X)=0$. 

(3) If $g(F)>1$, then $X$ is virtually symplectic with $v\kappa=1$.
\end{theorem}

By \cite{Agol,PW}, most irreducible $3$--manifolds are virtually fibered except some graph manifolds (including some Seifert fibered spaces) \cite{Neumann}. On the other hand, if $Y$ is not virtually fibered, and $\varphi\co Y\to Y$ is periodic, then by \cite{FV2} $Y\rtimes_{\varphi}S^1$ is not virtually symplectic.

This paper is organized as follows.  In Section~\ref{Sect:vb}, we prove Theorem~\ref{thm:vb}. Most of the argument is an application of Theorem~\ref{thm:vb3d}. When the fiber has a nontrivial JSJ decomposition, we apply results of Kojima \cite{Ko}. In Section~\ref{Sect:Kod}, we review the definition of symplectic Kodaira dimension, then we finish the proof of Theorem~\ref{thm:ClassifyFiber}. In Section~\ref{Sect:VS}, we prove Theorem~\ref{thm:VirtSympl} using Luttinger surgery. In Section~\ref{Sect:Red}, we discuss the case that the fiber is reducible. 

\ 

\noindent{\bf Acknowledgements}. 
The proof of Theorem~\ref{thm:ClassifyFiber} was written in 2010. Inspired by recent progress on related topics \cite{Bay,BF,FV5}, we expanded this note to the current version.
We wish to thank Chung-I Ho, Yi Liu and Stefano Vidussi for interesting discussion. We are also grateful to Anar Akhmedov, Inanc Baykur, Nikolai Saveliev and Stefano Vidussi for comments on earlier versions of this paper. The first author was supported
by NSF grant numbers DMS-1065927, DMS-1207037.   The second author was supported by an AIM Five-Year Fellowship, NSF grant
numbers DMS-1021956, DMS-1103976, and an Alfred P. Sloan Research Fellowship.

\section{Virtual Betti number}\label{Sect:vb}

\subsection{Preliminary on mapping tori}\label{Sect:MT}

We will use two methods to construct finite covers of a mapping torus $X$.

The first method is obvious:
The mapping torus of $f^k\co Y\to Y$ is a cyclic cover of the mapping torus of $f\co Y\to Y$.

The second method requires the construction of finite covers of $Y$.

\begin{lemma}\label{lem:FindCov}
Suppose $X$ is a mapping torus with fiber $Y$, and $\pi_1(Y)$ is finitely generated. Suppose $\widetilde Y$ is a finite cover of $Y$, then there is a finite cover $\widetilde X$ of $X$, such that $\widetilde X$ fibers over the circle with fiber $\widetilde Y$.
\end{lemma}
\begin{proof}
Suppose $f\co Y\to Y$ is the monodromy of $X$, $f_*\co \pi_1(Y)\to \pi_1(Y)$ is the induced map. Let $d=[\pi_1(Y):\pi_1(\widetilde Y)]$. Since $\pi_1(Y)$ is finitely generated, it has only finitely many index $d$ subgroups. So there exists an $n\in\mathbb N$ such that $f_*^n(\pi_1(\widetilde Y))=\pi_1(\widetilde Y)$. Let $X_n$ be the $n$--fold cyclic cover of $X$ dual to $Y$, then
$$\pi_1(X_n)=\langle\pi_1(Y),t|\:txt^{-1}=f_*^n(x), \forall x\in\pi_1(Y)\rangle.$$
Let $\widetilde X$ be the cover of $X_n$ corresponding to the subgroup generated by $\pi_1(\widetilde Y)$ and $t$. Since the conjugation by $t$ fixes $\pi_1(\widetilde Y)$ setwise, we conclude that $$\pi_1(\widetilde X)=\langle\pi_1(\widetilde Y),t|\:tyt^{-1}=f_*^n(y), \forall y\in\pi_1(\widetilde Y)\rangle.$$
$\widetilde X$ is the cover we want.
\end{proof}

The following observation is useful in our proof.

\begin{lemma}\label{lem:FPBundle}
Suppose $Y=F\rtimes B$, $f\co Y\to Y$ is a fiber-preserving map, hence $f$ induces a map $\overline f\co B\to B$. Then $Y\rtimes_f S^1$ is an $F$--bundle over $B\rtimes_{\overline f}S^1$.
\end{lemma}
\begin{proof}
Let $p\co Y\to B$ be the fibration of $Y$. Since $f$ is fiber-preserving, we have
\begin{equation}\label{eq:pf=fp}
p\circ f=\overline f\circ p.
\end{equation}
 The mapping tori of $f,\overline f$ are
$$M=Y\times[0,1]/(x,1)\sim(f(x),0),\quad\overline M=B\times[0,1]/(y,1)\sim(\overline f(y),0).$$
Using (\ref{eq:pf=fp}), we can verify that the fibration $$p\times\mathrm{id}\co Y\times[0,1]\to B\times[0,1]$$ induces a fibration
$$M\to \overline M,$$
which is an $F$--bundle.
\end{proof}

In the rest of this section, $X^4=Y\rtimes S^1$, and $Y$ is prime. By the Geometrization Theorem, either $Y$ is geometric or $Y$ has a nontrivial JSJ decomposition. If $Y$ is geometric, namely, $Y$ supports one of the eight Thurston geometries,
then $Y$ is either covered by a torus bundle over $S^1$ with Anosov monodromy, or a Seifert fibered space, or hyperbolic. Below we will discuss these cases.

\subsection{Quotients of torus bundles}\label{subsect:TB}

If $Y$ is covered by a torus bundle over $S^1$, then any finite cover of $Y$ is also covered by a torus bundle over $S^1$. So any finite cover of $X$ is covered by a mapping torus whose fiber is a torus bundle over $S^1$.
It follows that $vb_1(X)\le4$.

\subsection{Hyperbolic manifolds}\label{subsect:hyp}

If $Y$ is hyperbolic, then the mapping class group of $Y$ is finite. So $X$ is covered by $Y\times S^1$. By Theorem~\ref{thm:vb3d} we know $vb_1(X)=\infty$.

\subsection{Seifert fibered spaces}\label{sect:SFS}

The following fact can be found in \cite{Sc}.

\begin{prop}\label{prop:SFScover}
Any Seifert fibered space is finitely covered by a circle bundle over an oriented surface.
\end{prop}

The next lemma is elementary.

\begin{lemma}\label{lem:S1S2}
If $Y=S^1\times S^2$, then $X=Y\rtimes S^1$ is covered by $S^2\times T^2$.
\end{lemma}
\begin{proof}
After iterating the monodromy $f$ we may assume $f_*=\mathrm{id}$ on $H_2(Y)$, then $f$ is isotopic to the identity, (see, for instance, Lemma~\ref{lem:fp}). Hence $X$ is covered by $S^2\times T^2$.
\end{proof}

The following theorem is well known, see the Theorems~3.8, 3.9 and the discussion in the end of Section~3 in \cite{Sc}.

\begin{theorem}\label{thm:FiberPr}
Suppose that $M$ is a compact orientable Haken Seifert fibered space whose base has negative orbifold Euler characteristic. Then the Seifert fibration of $M$ is  unique up to isomorphism, and any homeomorphism on $M$ is isotopic
to a fiber-preserving homeomorphism.
\end{theorem}

\begin{prop}\label{prop:CovByBundle}
Suppose that $Y$ is an orientable Seifert fibered space over an orbifold $\mathcal B$. Let $\chi_{\mathrm{orb}}(\mathcal B)$ be the orbifold Euler characteristic of $\mathcal B$. Suppose that $f\co Y\to Y$ is an orientation preserving homeomorphism, $X=Y\rtimes_fS^1$.
Then there are three cases:\newline
(1) If $\chi_{\mathrm{orb}}(\mathcal B)>0$, then $vb_1(X)=1$ or $2$;
\newline(2) If $\chi_{\mathrm{orb}}(\mathcal B)=0$, then $vb_1(X)\le4$;
\newline(3) If $\chi_{\mathrm{orb}}(\mathcal B)<0$, then $vb_1(X)=\infty$.
\end{prop}
\begin{proof}
If $\chi_{\mathrm{orb}}(\mathcal B)>0$, then $Y$ is covered by $S^1\times S^2$ or $S^3$. If $Y$ is covered by $S^1\times S^2$, Lemma~\ref{lem:S1S2} implies that $X$ is covered by $S^2\times T^2$, so $vb_1(X)=2$. If $Y$ is covered by $S^3$, then $vb_1(X)=1$. 

If $\chi_{\mathrm{orb}}(\mathcal B)=0$, then $Y$ is finitely covered by a torus bundle over $S^1$. So $X$ is covered by an iterated torus bundle. The discussion in Subsection~\ref{subsect:TB} shows that $vb_1(X)\le4$.

Now we consider the case $\chi_{\mathrm{orb}}(\mathcal B)<0$.
By Lemma~\ref{lem:FindCov} and Proposition~\ref{prop:SFScover}, we may assume $Y$ is a circle bundle over an oriented surface $B$ with negative Euler characteristic.
By Theorem~\ref{thm:FiberPr}, $f$ is isotopic to a fiber-preserving homeomorphism. Let $\overline f\co B\to B$ be the map on $B$ induced by $f$.
  Lemma~\ref{lem:FPBundle} then implies that a finite cover of $X$ is a circle bundle over $\overline{X}$, the mapping torus of $\overline f$. Since $\chi(B)<0$, $vb_1(\overline X)=\infty$ by Theorem~\ref{thm:vb3d}. So $vb_1(X)=\infty$.
\end{proof}

\subsection{Irreducible manifolds with nontrivial JSJ decomposition}\label{Sect:JSJ}

Throughout this subsection, $Y$ is an orientable irreducible manifold with nontrivial JSJ decomposition, $f\co Y\to Y$ is an orientation preserving homeomorphism,  $X=Y\rtimes_fS^1$.

\begin{theorem}\label{thm:vbJSJ}
$Y,f,X$ are as above, then $vb_1(X)=\infty$.
\end{theorem}

Before we proceed, we remark that by Mostow's Rigidity, any homeomorphism of a complete hyperbolic $3$--manifold with finite volume is isotopic to a periodic map, namely, some iteration of this homeomorphism is isotopic to the
identity map.

By the standard JSJ theory, $f$ can be isotoped to send each JSJ piece to a JSJ piece. After iterating $f$, we may assume $f$ satisfies the following
\begin{condition}\label{cond:PrJSJ}
The monodromy $f$ sends each JSJ piece and each JSJ torus to itself, and the restriction of $f$ to each hyperbolic piece is the identity.
\end{condition}

\begin{lemma}\label{lem:OrBase}
$X$ is finitely covered by a manifold $\widetilde X=\widetilde Y\rtimes S^1$, such that each JSJ piece of $\widetilde Y$ is either hyperbolic or a Seifert fibered space over an orientable orbifold with negative Euler characteristic.
\end{lemma}
\begin{proof}
By \cite[Theorem~2.6]{Luecke}, $Y$ is finitely covered by a manifold $\widetilde Y$ as in the statement of the lemma. Our conclusion follows from Lemma~\ref{lem:FindCov}.
\end{proof}

Now we can work with $\widetilde Y$ instead of $Y$. Iterating $f$ again to ensure that it satisfies Condition~\ref{cond:PrJSJ} and $f$ induces an orientation preserving map on the base of each Seifert fibered piece. Let  $T_1,\dots,T_e$ be the JSJ tori. 

\begin{lemma}\label{lem:TorusId}
The restriction of $f$ to each JSJ torus $T_j$ is isotopic to the identity. So the mapping torus of $f|T_j$ is $T^3$.
\end{lemma}
\begin{proof}
If a JSJ torus $T_j$ is adjacent to a hyperbolic piece, then it follows from Condition~\ref{cond:PrJSJ} that $f|T_j$ is isotopic to the identity. If both sides of $T_j$ are Seifert fibered pieces, then
the Seifert fibers from the two sides are not parallel on $T_j$, otherwise we could glue the Seifert fibrations on the two sides together hence $T_j$ would not be a JSJ torus. Since $f$ is orientation preserving and the induced map on the base is
also orientation preserving, $f$ preserves the orientation of the Seifert fibers. Since two Seifert fibers from two sides of $T_j$ are linearly independent in $H_1(T_j;\Q)$, $f$ induces the identity on $H_1(T_j;\Q)$, hence is isotopic to the identity.
\end{proof}

Let $R_j$ be the mapping torus of $f|T_j$, $R=\cup R_j$. Suppose $R_j\subset \partial X_i$.

\begin{lemma}\label{lem:TwoR}
There exists a finite cover $\rho:\widetilde X\to X$, so that there are two components $R_j',R_j''$ of $\rho^{-1}(R_j)$, such that $\widetilde X-(R_j'\cup R_j'')$ is connected.
\end{lemma}
\begin{proof}
By \cite[Propositions 5 and 7]{Ko}, there exists a finite cover $\pi\co\widetilde Y\to Y$, so that there are two components $T_j',T_j''$ of $\pi^{-1}(T_j)$, such that $\widetilde Y-(T_j'\cup T_j'')$ is connected.
By Lemma~\ref{lem:FindCov}, $X$ is covered by a mapping torus $\widetilde X=\widetilde Y\rtimes_{\widetilde f} S^1$. Iterating the monodromy $\widetilde f$ if necessary, we may assume $\widetilde f(T_j')=T_j', \widetilde f(T_j'')=T_j''$. Let $R'_j,R''_j$ be the mapping tori of $\widetilde f|T_j', \widetilde f|T_j''$, then $R'_j,R''_j$ are components of the preimage of $R_j$, and $\widetilde X-(R_j'\cup R_j'')$, being $(\widetilde Y-(T_j'\cup T_j''))\rtimes S^1$, is connected.
\end{proof}

Theorem~\ref{thm:vbJSJ} easily follows from Lemma~\ref{lem:TwoR}. In fact, the $n$--fold cyclic cover of $\widetilde X$ with respect to $R_j'$ will have $b_1\ge n+1$.

\subsection{Proof of Theorem ~\ref{thm:vb}}
If $Y$ is prime, then the Geometrization Theorem shows that either $Y$ is geometric or $Y$ has a nontrivial JSJ decomposition.
Now Theorem~\ref{thm:vb} follows from 
Theorem~\ref{thm:vbJSJ}, 
Propositions~\ref {prop:CovByBundle}, and the discussions in Subsection~\ref{subsect:TB} and Subsection~\ref{subsect:hyp}.

\section{Symplectic mapping tori}\label{Sect:Kod}

\subsection{Constructing symplectic structures}

We begin with a  construction of symplectic structures on a general class of mapping tori. 

\begin{definition}
Let $f\co Y\to Y$ be an orientation preserving homeomorphism of a closed, oriented, connected 3-manifold $Y$. We say that the pair $(Y, f)$ is {\it fibered} if
$Y$ admits a fibration over the circle such that $f$ preserves the homology class of the fiber.  
\end{definition}

\begin{lemma}\label{lem:fp} 
If $(Y, f)$ is fibered, then $f$ is isotopic to a fiber preserving map that preserves the orientation of the fibers.
 \end{lemma}
\begin{proof} Suppose $(Y, f)$ is fibered with respect to a fibration $p\co Y\to S^1$ with $F$ as a fiber. 
 Since $f_*([F])=[F]$, there is an ambient isotopy of $Y$ which takes $f(F)$ to $F$. Hence $f$ can be isotoped so that $f(F)=F$. A further isotopy will make $f$ a fiber-preserving map with respect to the fibration $p$ and $f$ preserves the orientation of the fibers.
 \end{proof}

\begin{prop}\label{prop:MTSympl}
Every mapping torus $X$  with $(Y, f)$ fibered is symplectic.
\end{prop}
\begin{proof} By Lemma~\ref{lem:fp}   $X$ fibers over $T^2$.  Now the statement follows from \cite{ThSympl}
if the fiber genus of $Y$ is not equal to one, and it follows from \cite{Geiges} if the fiber genus is equal to 1.
\end{proof}

\subsection{The torus bundle case}

In this subsection, we study the case that $Y$ is covered by a $T^2$--bundle over the circle.

\begin{lemma}\label{lem:PrimClass}
Suppose that $Y$ is covered by a $T^2$--bundle over the circle, $\alpha\in H_2(Y;\Z)$ is a primitive homology class. Then $Y$ admits a $T^2$--fibration over $S^1$ such that $\alpha$ is represented by a fiber.
\end{lemma}
\begin{proof}
We first consider the case that $Y$ itself is a $T^2$--bundle over the circle.
Let $\varphi$ be the monodromy of $Y$, $F$ be a fiber of $Y$. Consider $$k_1=\mathrm{rank}\:\mathrm{ker}(\varphi_*-\mathrm{id}\co H_1(T^2)\to H_1(T^2)).$$ If $k_1=0$,
then $H_2(Y)$ is generated by $[F]$, our conclusion obviously holds. If $k_1=2$, then $Y=T^3$, our conclusion also holds.

From now on we assume $k_1=1$. Since $H_1(T^2)$ is torsion-free, there exists a simple closed curve $c\subset F$  representing a generator of $\mathrm{ker}(\varphi_*-\mathrm{id})$. We may isotope $\varphi$ so that $\varphi(c)=c$. 
The complement of $c\times S^1$ is an annulus bundle over $S^1$, hence homeomorphic to $T^2\times I$. So $c\times S^1$ is also a fiber of a fibration of $Y$.

$H_2(Y)\cong\Z^2$ is generated by $[c\times S^1]$ and $[F]$.
Suppose $\alpha=p[c\times S^1]+q[F]$. Without loss of generality, we may assume $p,q>0$. A surface representing $\alpha$ can be obtained as follows. We take $p$ copies of $c\times S^1$ and $q$ copies of $F$, make them transverse,
then perform oriented cut-and-paste to them. The resulting surface is a torus $F'$ representing $\alpha$. The complement of $q$ copies of $F$ is $q$ copies of $T^2\times I$, and each $c\times S^1$ intersects each $T^2\times I$ in
a vertical annulus, so we see that the complement of $F'$ is homemorphic to $T^2\times I$. Hence $F'$ is a fiber of a $T^2$--fibration of $Y$. This finishes the proof when  $Y$ itself is a $T^2$--bundle over $S^1$.

If $Y$ is covered by $\widetilde Y=T^2\rtimes S^1$, let $p\co \widetilde Y\to Y$ be the covering map. Since the Thurston norm of $\widetilde Y$ is zero, it follows from the fact that the singular Thurston norm is equal to the Thurston norm \cite{G1} that the Thurston norm of $\alpha$ is zero, so $\alpha$ is represented by a collection of disjoint tori. Since $\alpha$ is primitive, it is easy to see $\alpha$ is represented by an embedded torus $T$. The preimage $p^{-1}T$ consists of several disjoint homologically essential tori, so the case we discussed before implies that $\widetilde Y-p^{-1}(T)$ consists of several copies of $T^2\times I$. Now $Y-T$ is a connected manifold covered by $T^2\times I$, and it has two boundary components each homeomorphic to $T^2$, so $Y-T=T^2\times I$. Hence $Y$ is a torus bundle with fiber $T$. 
\end{proof}

\begin{prop}\label{prop:TorusBundle}
Suppose that $Y$ is covered by a $T^2$--bundle over the circle, $f\co  Y\to Y$ is an orientation-preserving homeomorphism, $X=Y\rtimes_fS^1$. Then the following 4 conditions are equivalent: \newline
(1) $X$ admits a symplectic structure;\newline
(2) $b_1(X)\ge2$;\newline
(3) there exists a $T^2$--fibration of $Y$ such that $f$ is fiber preserving and $f$ preserves the orientation of fibers;\newline
(4) $X$ is an orientable $T^2$--bundle over $T^2$.
\end{prop}
\begin{proof}
\noindent(1)$\Rightarrow$(2). Let
$$k_i=\mathrm{rank}\:\mathrm{ker}(f_*-\mathrm{id}\co H_i(Y)\to H_i(Y)).$$
By Poincar\'e duality $k_1=k_2$. Moreover, by Mayer--Vietoris it is easy to see $b_1(X)=k_1+1$, $b_2(X)=2k_1$.
Since $X$ is symplectic, $b_2(X)>0$, so $k_1\ge1$, hence $b_1(X)\ge2$.

\noindent(2)$\Rightarrow$(3).
Since $b_1(X)\ge2$, $k_2=k_1\ge1$. Let $\alpha\in H_2(Y)$ be a primitive class in $\mathrm{ker}(f_*-\mathrm{id})$. 
By Lemma~\ref{lem:PrimClass},
$\alpha$ represents a fiber $F$ of a $T^2$--fibration $p\co Y\to S^1$. Our conclusion follows from Lemma~\ref{lem:fp}.

\noindent(3)$\Rightarrow$(4).
This follows from Lemma~\ref{lem:FPBundle}.

\noindent(4)$\Rightarrow$(1). This is a theorem of Geiges \cite{Geiges}.
\end{proof}

\subsection{Symplectic Kodaira dimension and the virtual extension}\label{Subsect:Kod}

A symplectic $4-$manifold $(X, \omega)$ is said to be minimal if it does not
contain any symplectic sphere with self-intersection $-1$. 

When $(X, \omega)$ is minimal, its symplectic Kodaira dimension  is
defined by the products $K_{\omega}^2$ and $K_{\omega}\cdot [\omega]$,
where $K_{\omega}$ is the symplectic canonical class:

$$\kappa(X, \omega)=\left\{\begin{array}{ll}
-\infty & K_{\omega}^2<0\ or\ K_{\omega}\cdot [\omega] <0\\
0 & K_{\omega}^2=0\ and\ K_{\omega}\cdot [\omega]=0 \\
1 & K_{\omega}^2=0\ and\ K_{\omega}\cdot [\omega]>0\\
2 & K_{\omega}^2>0\ and\ K_{\omega}\cdot [\omega] >0
                  \end{array}\right.$$

For a general symplectic  $4-$manifold, the Kodaira dimension  is
defined as the Kodaira dimension of any of its minimal models.

According to \cite{L1},
$\kappa(X, \omega)$ is  independent of the choice
of symplectic form $\omega$ and hence it will be denoted by $\kappa(X)$.

If $(X,\omega)$ is symplectic, $p\co \widetilde X\to X$ is a finite degree covering map, then $\widetilde X$ has a symplectic form $\widetilde{\omega}=p^*\omega$, and $K_{\widetilde{\omega}}=p^*K_{\omega}$. It was observed in \cite{LZ}  that $\kappa(X)=\kappa(\widetilde X)$. In light of this invariance property, we introduce  the following virtual version of $\kappa$.

\begin{definition} Suppose $X$ is virtually symplectic, let  $\widetilde X$ be any symplectic manifold which finitely covers $X$.  We  define the 
{\it virtual Kodaira dimension} of $X$ by 
$$v\kappa(X)=\kappa(\widetilde X).$$ 
\end{definition}

It is easy to check  that $v\kappa(X)$ is independent of the choice of $\widetilde X$.

\subsection{Symplectic  mapping tori with non-positive $\kappa$}
In this subsection $Y$ is a prime $3-$manifold. 

We first make a simple observation. 


\begin{lemma}\label{lem:Kodle1}
If $X$ is a mapping torus admitting a symplectic structure $\omega$, then $\kappa(X)\leq 1$.
\end{lemma}
\begin{proof}
Let $X$ be a mapping torus with fiber $Y$. 
We first show that $(X, \omega)$ is always minimal. 
Clearly $\pi_2(X)=\pi_2(Y)$. If $(X, \omega)$ is not minimal then $\pi_2(Y)$ is infinite. If $Y$ has infinite $\pi_2$, since $Y$ is assumed to be prime, $Y=S^2\times S^1$. And $X=S^2\times T^2$, which is
clearly minimal. 

Notice that the Euler number of $X$ is zero, and the signature of $X$ is zero. We have $K_{\omega}^2=3\sigma(X)+2\chi(X)=0$, so $\kappa(X)\ne2$
since $(X, \omega)$ is minimal.
\end{proof}

\begin{proof}[Proof of Theorem~\ref{thm:ClassifyFiber}]


Suppose a mapping torus $X$ is symplectic and $\kappa(X)=-\infty$. By the classification in \cite{Liu},  $X$ must be an $S^2-$bundle over $T^2$.
From the homotopy exact sequence, $\pi_1(Y)$ is a subgroup of $\Z\oplus \Z$ with quotient $\Z$. It is clear that $\pi_1(Y)=\Z$.
The only such manifold is $S^1 \times S^2$. Then $X=S^2\times T^2$.

If $\kappa=0$, then by \cite{L2,B} $2\le vb_1(X)\le4$. By Theorem~\ref{thm:vb}, $Y$ is covered by a $T^2$--bundle over $S^1$. Proposition~\ref{prop:TorusBundle} shows that $X$ is a $T^2$--bundle over $T^2$.
\end{proof}

Minimal symplectic $4$--manifolds with $\kappa=0$ have torsion symplectic canonical class (\cite{L1}), and thus can be viewed as symplectic analogues of Calabi--Yau surfaces.
It is shown in \cite{L2}  that symplectic CY surfaces are $\Z$--homology K3 surface, $\Z$--homology Enriques surface, and $\Q$--homology
$T^2$--bundles over $T^2$.

A basic problem is whether a symplectic CY surface must be diffeomorphic to K3 surface, Enriques surface or a $T^2$--bundle over $T^2$.

It is shown in \cite{U} and \cite{Do} respectively that nontrivial positive genus and genus zero fiber sums do not give rise to  any new symplectic CY surface.

Friedl and Vidussi \cite{FV3} (and \cite{Bow}) investigate this problem for circle bundles over 3-manifolds and deduce the base must be a $T^2-$bundle. Furthermore,
 the total space is shown to be finitely covered by a $T^2-$bundle over $T^2$, and
if the circle bundle has trivial or non-torsion Euler class,
it is itself a $T^2-$bundle over $T^2$.

It is recently observed   in \cite{Bay}  that, if a symplectic CY surface $(X, \omega)$ fibers over a $2-$manifold, then $X$ is a torus bundle over 
torus.  

\begin{remark}As suggested by Saveliev, it might be  possible to  classify complex mapping tori. One interesting example is the Inoue surface, which, 
topologically, is a $3-$torus bundle over $S^1$ with infinite order monodromy and with homology of $S^1 \times S^3$.
\end{remark}

\section{Virtual symplecticity}\label{Sect:VS}

The goal of this section is to prove Theorem~\ref{thm:VirtSympl}.

 In case (1) of Theorem~\ref{thm:VirtSympl}, $Y$ is either $S^1\times S^2$ or $\mathbb RP^3\#\mathbb RP^3$. By Lemma~\ref{lem:S1S2}, $X$ is always covered by $S^2\times T^2$, so $X$ is virtually symplectic and $v\kappa(X)=-\infty$.
The case (2) of Theorem~\ref{thm:VirtSympl}
 immediately follows from Proposition~\ref{prop:TorusBundle}.

From now on we consider case (3) of Theorem~\ref{thm:VirtSympl}, namely, $X$ is virtually fibered with genus of fiber $>1$.

Let $(X, \omega)$ be a symplectic $4-$manifold  with a Lagrangian torus $L$.
By Weinstein's Lagrangian neighborhood theorem, there is a canonical Lagrangian framing $f$ of $L$. The corresponding push off
is called the  Lagrangian push off.  Denote the meridian of $L$ by $m_L$.  Given a simple loop $\gamma$ in $L$ and an integer $k$, the {\it Luttinger surgery} is the 
$f-$framed torus surgery 
whose regluing map is specified by sending $m_L$ to a simple loop  in the homology class  $[m_L]+k[\gamma_f]$. 
It was discovered  by Luttinger in \cite{Lut} that the resulting manifold admits symplectic structures. 

We will apply Luttinger surgery to the case that $X=F\times (S^1\times S^1)$ with a product symplectic form,  and $L$ is a
product Lagrangian torus of the form 
$$\alpha\times p \times S^1     \quad\mathrm{or} \quad  \beta\times S^1\times p,$$
where $F$ is a surface, $\alpha, \beta$ are simple loops in $F$, and $p$ is a point in $S^1$.   

In this case the Lagrangian framing is easy to describe. 
In particular, the Lagrangian push offs   are still in the same product form. 

Luttinger surgery is related to Dehn surgery in dimension $3$. Suppose $M$ is a $3$--manifold, $K\subset M$ is a knot with a frame $\lambda$ and meridian $\mu$. For any rational number $\frac pq$,  let $M_{\frac pq}(K)$ be the manifold obtained from $M$ by doing Dehn surgery on $K$ with slope in the homology class $p[\mu]+q[\lambda]$. Now suppose $M\times S^1$ is a submanifold of a symplectic $4$--manifold $X$, and $K\times S^1$ is a Lagrangian submanifold such that $\lambda\times S^1$ is the Lagrangian framing. Then $M_{1/k}\times S^1$ is obtained from $M\times S^1$ by a Luttinger surgery on $K\times S^1$ whose regluing map sends $m_L$ to $[m_L]+k[\lambda\times\mathrm{point}]$.

For any simple closed curve $\alpha$ in a surface $\Sigma$, let  $\tau_{\alpha}$ be the positive Dehn twist along $\alpha$.
The following fact is well-known.

\begin{prop}\label{prop:DehnTwist}
Suppose $\Sigma$ is a surface in a $3$--manifold $M$, $K\subset \Sigma$ is a knot. Let $\lambda$ be the frame on $K$ specified by $\Sigma$, then $M_{1/k}(K)$ can also be obtained from $M$ by cutting open along $\Sigma$ then regluing by $\tau_{K}^k$.
\end{prop}

Suppose $Y$ is finitely covered by a surface bundle $\widetilde Y=F\times_{\varphi}S^1$ with $g(F)>1$. By Nielsen--Thurston's classification of surface automorphisms, there exists a collection of disjoint essential simple closed curve $\mathcal C\subset F$, such that $\varphi(\mathcal C)$ is isotopic to $\mathcal C$. Moreover, $F\backslash\mathcal C$ has two parts $F_1, F_2$, such that after an isotopy $\varphi(F_1)=F_1$, $\varphi(F_2)=F_2$, $\varphi_1=\varphi|{F_1}$ is freely isotopic to a pseudo--Anosov map, and $\varphi_2=\varphi|_{F_2}$ is freely isotopic to a periodic map. Iterating $\varphi$ if necessary, we may assume $\varphi|_{\mathcal C}=\mathrm{id}_{\mathcal C}$, $\varphi_1$ maps each component of $F_1$ to itself,  and $\varphi_2$ is freely isotopic to $\mathrm{id}_{F_2}$. Let $Y_i=F_i\rtimes_{\varphi_i}S^1$, then $Y_1$ is hyperbolic and $Y_2=F_2\times S^1$.

Now $X$ is finitely covered by $\widetilde X=\widetilde Y\rtimes_{\Psi}S^1$. The monodromy $\Psi$ can be isotoped to a map which sends each hyperbolic JSJ piece of $\widetilde X$ to a hyperbolic piece, and each Seifert fibered piece to a Seifert fibered piece. Iterating $\Psi$ if necessary, we may assume $\Psi_1=\Psi|_{Y_1}=\mathrm{id}$ and $\Psi_2=\Psi_{Y_2}$ maps each component of $Y_2$ to itself. We may also assume $\Psi_2$ preserves the $S^1$--fibers in $Y_2=F_2\times S^1$, so it induces a map $\psi_2$ on $F_2$. By Lemma~\ref{lem:TorusId}, the restriction of $\Psi$ to each boundary component of $Y_2$ is isotopic to the identity. Let $X_i=Y_i\rtimes_{\Psi_i} S^1$, then
$$X_1=(F_1\rtimes_{\varphi_1}S^1)\times S^1,\quad X_2=(F_2\times S^1)\rtimes_{\Psi_2}S^1.$$

The proof of Theorem~\ref{thm:VirtSympl} will be completed by the next proposition.

\begin{prop}\label{prop:Omega}
There exists a symplectic form $\Omega$ on $\widetilde X$ such that $K_{\Omega}[F]=2g(F)-2$ and $\kappa(\widetilde X)=1$.
\end{prop}
\begin{proof}
We start with the manifold $F\times S^1\times S^1$, which is clearly symplectic. Then we try to reconstruct $\widetilde X$ by doing Luttinger surgeries on a collection of disjoint Lagrangian tori in $F\times S^1\times S^1$. 

The automorphism $\varphi$ is isotopic to a product of (positive or negative) Dehn twists. 
Since $\varphi|_{F_2}=\mathrm{id}_{F_2}$, we may assume $$\varphi=\tau^{k_1}_{\alpha_1}\circ\tau^{k_2}_{\alpha_2}\circ\cdots\circ\tau^{k_{n_1}}_{\alpha_{n_1}},$$
where $\alpha_i$, $i=1,2,\dots,n_1$, are simple closed curves supported outside the interior of $F_2$  \footnotemark, $k_i\in\mathbb Z$. Choose successive points $p_1,p_2\dots,p_{n_1}\in S^1$.
By Proposition~\ref{prop:DehnTwist}, the manifold $$M_1=(F\rtimes_{\varphi}S^1)\times S^1$$ can be obtained from $(F\times S^1)\times S^1$ by Luttinger surgeries on the Lagrangian tori
\begin{equation}\label{eq:ToriA}
\alpha_1\times p_1\times S^1, \dots,\alpha_{n_1}\times p_{n_1}\times S^1,
\end{equation}
where the regluing map on $ \alpha_i\times p_i\times S^1$ sends the meridian $m^1_i$ to a curve in the homology class $[m^1_i]+k_i[\alpha_i]$. We see that $M_1$ is the union of $X_1$ and $(F_2\times S^1)\times S^1$.

\footnotetext{The curves $\alpha_i$ may not be supported in $F_1$. Some of them may be components of $\mathcal C$ which are adjacent to $F_2$ on both sides.}

Suppose that $$\psi_2=\tau^{l_1}_{\beta_1}\circ\tau^{l_2}_{\beta_2}\circ\cdots\circ\tau^{l_{n_2}}_{\beta_{n_2}},$$
where $\beta_j\subset F_2$ are simple closed curves, $l_j\in\mathbb Z$. Choose successive points $q_1,q_2\dots,q_{n_2}\in S^1$.
In $M_1\supset (F_2\times S^1)\times S^1$, we do further Luttinger surgeries on the Lagrangian tori
\begin{equation}\label{eq:ToriB}
\beta_1\times S^1\times q_1,\dots,\beta_{n_2}\times S^1\times q_{n_2}
\end{equation}
where the regluing map on $ \beta_j\times S^1\times q_j$ sends the meridian $m^2_j$ to a curve in the homology class $[m^2_j]+l_i[\beta_j]$. 
By Proposition~\ref{prop:DehnTwist}, the resulting manifold $M_2$ is a union of $X_1$ with $S^1\times(F_2\rtimes_{\psi_2}S^1)$.

Consider $X_2=(F_2\times S^1)\rtimes_{\Psi_2}S^1$, which also has a circle bundle structure $X_2=S^1\rtimes(F_2\rtimes_{\psi_2}S^1)$. Since $F_2$ has no closed components and the restriction of the bundle on the bundary of $F_2\rtimes_{\psi_2}S^1$ is trivial, the Poincar\'e dual of the Euler class $e(X_2)$ of this bundle can be represented by a (possibly disconnected) closed curve in a fiber.
The difference between the two circle bundles $S^1\times(F_2\rtimes_{\psi_2}S^1)$ and $X_2$ is a twisting in the circle direction on the Poincar\'e dual of $e(X_2)$.
Let $q_0\in S^1\backslash\{q_1,\dots,q_{n_2}\}$ be a point, and $\beta_0\subset F_2$ be a (possibly disconnected) closed curve such that $$\beta_0\times q_0\subset F_2\rtimes_{\psi_2}S^1$$ represents this Poincar\'e dual. 
Then $\widetilde X$ can be obtained from $M_2\supset S^1\times(F_2\rtimes_{\psi_2}S^1)$ by doing  Luttinger surgeries on 
\begin{equation}\label{eq:ToriB0}
\beta_0\times S^1\times q_0
\end{equation}
where the gluing map on each component $\beta_0^l\times S^1\times q_0$ sends the meridian $m_l$ to a curve in the homology class $[m_l]+[\mathrm{point}\times S^1\times q_0]$.

Now $\widetilde X$ is obtained from $F\times T^2$ by doing Luttinger surgeries on disjoint Lagrangian tori in (\ref{eq:ToriA}), (\ref{eq:ToriB}), (\ref{eq:ToriB0}). So it has a symplectic form $\Omega$.
Let $p\in S^1\backslash\{p_1,\dots,p_{n_1}\}$, $q\in S^1\backslash\{q_0,q_1,\dots,q_{n_2}\}$, then $F\times p\times q\subset F\times T^2$ is a symplectic surface disjoint from the previous Lagrangian tori, and $K_{F\times T^2}[F\times p\times q]=2g(F)-2$. Since the symplectic structure is unchanged outside a neighborhood of the Lagrangian tori, in $\widetilde X$ we also have $K_{\Omega}[F\times p\times q]=2g(F)-2$.

Finally, by Lemma~\ref{lem:Kodle1} and Theorem~\ref{thm:ClassifyFiber}, $\kappa(\widetilde X)=1$. Hence $v\kappa(X)=1$. 
\end{proof}

We remark  that,  whenever $Y$ is virtually fibered and $X$ is virtually symplectic, by Theorem~\ref{thm:VirtSympl}  we have $v\kappa(X)=\kappa^t(Y)$. Here $\kappa^t$ is the topological Kodaira dimension of $3-$manifolds introduced in \cite{Z} by Weiyi Zhang.

\section{Discussion on reducible fibers}\label{Sect:Red}

When the fiber is reducible, the monodromy is more complicated than the case of irreducible $3$--manifolds \cite{McCull}. We make the following two conjectures.

\begin{conj}\label{conj:vbRed}
Suppose $X=Y\rtimes S^1$ is a $4$--manifold. If $Y$ is reducible, then $vb_1(X)=\infty$ unless $Y=S^2\times S^1$ or $\mathbb RP^3\#\mathbb RP^3$.
\end{conj}

\begin{conj}\label{conj:SymplRed}
Suppose $X=Y\rtimes S^1$ is a sympectic $4$--manifold. If $Y$ is reducible, then $Y=S^2\times S^1$ and $X=S^2\times T^2$.
\end{conj}

\begin{definition}
A group $G$ is {\it residually finite}, if for any nontrivial element $\alpha\in G$, there exists a finite index normal subgroup $H\vartriangleleft G$, such that $\alpha\notin H$.
\end{definition}

It is well known that the fundamental group of a Haken $3$--manifold is residually finite \cite{He2}. In fact, the Geometrization Conjecture implies that any $3$--manifold group is residually finite. 

The reason that in Conjecture~\ref{conj:vbRed} we exclude $S^2\times S^1$ and $\mathbb RP^3\#\mathbb RP^3$ is the following lemma.

\begin{lemma}\label{lem:CovSphere}
Suppose that $Y$ is a closed orientable reducible $3$--manifold, and $Y$ is not $S^2\times S^1$ or $\mathbb RP^3\#\mathbb RP^3$, then $Y$ has a finite cover of the form $(S^2\times S^1)\#(S^2\times S^1)\#Z$ for some $3$--manifold $Z$.
\end{lemma}
\begin{proof}
By our assumption, we may assume $Y=Y_1\#Y_2$, where $Y_1\ne S^3$ and $|\pi_1(Y_2)|>2$. By the residual finiteness of $3$--manifold groups, $Y_2$ has a finite cover $\widetilde Y_2$ of degree $d>2$. Hence $Y$ is $d$--fold covered by $dY_1\#\widetilde Y_2$. Again by the residual finiteness, there is a surjective map $\rho\co \pi_1(Y_1)\to G$ with $|G|<\infty$. We can construct a surjective map $\overline{\rho}\co \pi_1(dY_1\#\widetilde Y_2)\to G$, such that the restriction of $\overline{\rho}$ on $\pi_1(\widetilde Y_2)$ is trivial, and the  restriction of $\overline{\rho}$ on each of the $d$ $\pi_1(Y_1)$ factors is $\rho$. Let $\widetilde Y$ be the cover of $dY_1\#\widetilde Y_2$ corresponding to $\ker\overline{\rho}$. Let $S_1,\dots,S_{d-1}$ be separating spheres in $dY_1\#\widetilde Y_2$ which separates the $d$ punctured copies of $Y_1$, then each $S_i$ lifts to $|G|$ disjoint spheres. Let $\widetilde S_i$ be one of the lifts of $S_i$. Then the complement of $\widetilde S_1\cup\cdots\cup\widetilde S_{d-1}$ in $\widetilde Y$ is connected. Thus $\widetilde Y$ has a connected summand $(d-1)S^2\times S^1$.
\end{proof}

When the monodromy is the identity, Conjecture~\ref{conj:SymplRed} holds true by the work of McCarthy \cite{McC}.
Baykur and Friedl \cite{BF} proved Conjecture~\ref{conj:SymplRed} in the case that the monodromy preserves a separating essential sphere.
Below we give a plausible but not rigorous argument to prove Conjecture~\ref{conj:SymplRed}.

If $X$ is as in Conjecture~\ref{conj:SymplRed}, then by Lemmas~\ref{lem:FindCov} and \ref{lem:CovSphere} we may assume $Y$ has an $S^1\times S^2$ summand, so the Monopole or Heegaard Floer homology of $Y$ with certain twisted coefficients is zero \cite{Ni}. The monodromy induces a map on the Floer homology of $Y$ \cite{JT}. The Seiberg--Witten or mixed invariant of the mapping torus should be the Lefschetz number of the map on Floer homology \cite{Froy}, hence should be zero. By the work of Taubes \cite{Taubes1}, $X$ cannot be symplectic.

\end{document}